\documentclass[11pt]{amsart}
\usepackage{verbatim}
\usepackage{amsmath}
\usepackage{enumerate}
\usepackage{amssymb}
\usepackage{color}
\usepackage{url}

\usepackage{multirow}
\usepackage[normalem]{ulem} 

\newmuskip\pFqskip
\mathchardef\pFcomma=\mathcode`, 
\newcommand*\pFq[5]{%
  \begingroup
  \begingroup\lccode`~=`,
    \lowercase{\endgroup\def~}{\pFcomma\mkern\pFqskip}%
  \mathcode`,=\string"8000
  {}_{#1}F_{#2}\biggl(\genfrac..{0pt}{}{#3}{#4};#5\biggr)%
  \endgroup
}

\newtheorem{Theorem}{Theorem}[section]
\newtheorem{Lemma}[Theorem]{Lemma}

\theoremstyle{remark}

\newtheorem{Remark}[Theorem]{\bf Remark}

\renewcommand{\d}{{\mathrm d}}
\renewcommand{\Im}{\operatorname{Im}}

\begin{document}

\title{Holonomic alchemy and series for $1/\pi$}

\author{Shaun Cooper}
\address{Institute of Natural and Mathematical Sciences, Massey University\,---\,Albany, Private Bag 102904, North Shore Mail Centre, Auckland 0745, New Zealand}
\email{s.cooper@massey.ac.nz}

\author{James G.~Wan}
\address{Engineering Systems and Design, Singapore University of Technology and Design, 8 Somapah Rd.~487372, Singapore}
\email{james\_wan@sutd.edu.sg}

\author{Wadim Zudilin}
\address{School of Mathematical and Physical Sciences, The University of Newcastle, Callaghan NSW 2308, Australia}
\email{wzudilin@gmail.com}

\begin{abstract}
We adopt the `translation' as well as other techniques to express several identities
conjectured by Z.-W.~Sun by means of known formulas for $1/\pi$
involving Domb and other Ap\'ery-like sequences.
\end{abstract}

\maketitle
\allowdisplaybreaks
\numberwithin{equation}{section}

\section{Introduction}
\label{sec1}

The theory of Ramanujan's series for $1/\pi$ received a boost with the
announcement of a large number of conjectures by Z.-W.~Sun \cite{sun}. That work, which was
first published on \texttt{arXiv.org} on Feb.~28, 2011, has been expanded through 47 versions at the time of writing.
The conjectures have stimulated the development of new ideas, e.g., see \cite{chanwanzudilin,guillera,rogersstraub,wan,wanzudilin,zudilinBAMS}.
Despite the strong interest, a large number of conjectures remain open.

One of the goals of this work is to use a variety of methods to prove many of Sun's conjectures.
In particular, we use translation techniques to convert several of the
conjectures into known series that have already been classified.
We also offer short and alternative proofs for some of the conjectures that have already been
resolved, e.g., the `\$520 challenge' \cite[Eq.~(3.24)]{sun} that was first proved by M.~Rogers and A.~Straub~\cite{rogersstraub}.

All of the underlying generating
functions that we shall encounter are holonomic. That is, they are solutions of linear differential
equations with polynomial coefficients. We provide fairly full detail for the examples in the next two sections.
In subsequent sections we are more brief and just communicate the main results, as it is a matter of routine
to verify the computational details. In particular, we make frequent use, normally without
explanation, of the standard algorithms for
holonomic functions and their computer implementations, e.g., Maple's \texttt{gfun} package and the Wilf--Zeilberger algorithm.

Another goal is to classify the conjectures.
Although our work provides an identification of several of the conjectured series with known series,
a full classification remains elusive. Table~\ref{table2} provides a summary of the underlying series.
As a degree of mystery is still present, the topic is somewhat `alchemical' in nature.

\section{Conjectures (5.1)--(5.8): Level 10}
\label{sec2}

Conjectures (5.1)--(5.8) in \cite{sun} involve series of the form
\begin{equation}
\label{10p}
\sum_{n=0}^\infty \sum_{k=0}^\infty (an+b)\, h(n,k)\,x^{n+k}
\end{equation}
for particular values of $a$, $b$ and $x$, where
$$
h(n,k)={2n \choose n}{2k \choose k}{n+2k \choose n}{n \choose k}.
$$
Observe that $h(n,k)=0$ if $k>n$. We will prove three lemmas and use them to convert the series~\eqref{10p} to
an equivalent series that can be parameterized by level~$10$ modular forms.

\begin{Lemma}
\label{lemma1}
The following identity holds:
$$
\sum_{k=0}^n {n \choose k}^4
= \sum_{k=0}^n {2k \choose k}{2n-2k \choose n-k}{n+k \choose n-k}{n-k \choose k}.
$$
\end{Lemma}

\begin{proof}
It is routine to use a computer algebra system and apply (for instance) Zeilberger's algorithm \cite{aeqb} to show that each sequence satisfies the same three-term
recurrence relation and initial conditions. Alternatively, make the specialization
$$
b=c=-n,\quad d=-\frac{n}{2},\quad e=\frac12-\frac{n}{2}
$$
in Whipple's identity \cite[Theorem 3.4.5]{aar}
\begin{multline*}
\pFq{7}{6}{a,1+\frac{a}{2},b,c,d,e,-n}{\frac{a}{2},1+a-b,1+a-c,1+a-d,1+a-e,1+a+n}{1}
\\
=\frac{(1+a-d-e)_n(1+a)_n}{(1+a-d)_n(1+a-e)_n}\,
\pFq{4}{3}{1+a-b-c,d,e,-n}{1+a-b,1+a-c,d+e-a-n}{1}
\end{multline*}
and then take the limit as $a\rightarrow -n$.
\end{proof}

The next result will be used to compute derivatives. We call it a \textit{satellite identity}\,---\,the term we coin from~\cite{zudilinBAMS};
for details of why such identities exist and how to find them in a general situation see Remark~\ref{sat-rem} below.

\begin{Lemma}
\label{lemma2}
The following identity holds in a neighborhood of $x=0$:
$$
\sum_{n=0}^\infty \sum_{k=0}^\infty h(n,k)x^{n+k}\left\{4x-2n(1-x)+3k(1+4x)\right\}=0.
$$
\end{Lemma}

\begin{proof}
Each of the power series
$$
\sum_n\sum_k h(n,k)x^{n+k}, \quad
\sum_n\sum_k nh(n,k)x^{n+k}
\quad\text{and}\quad
\sum_n\sum_k kh(n,k)x^{n+k}
$$
satisfies a 4th order linear differential equation with coefficients from $\mathbb Z[x]$.
Such a differential equation can be produced by the multiple Wilf--Zeilberger algorithm.
It is routine to use a computer algebra system to verify that the desired number of leading coefficients
in the $x$-expansion of the left-hand side of the required equality are zero, thus giving the result.
\end{proof}

\begin{Lemma}
\label{fDf}
Let
$$
f(x) = \sum_{n=0}^\infty \left\{\sum_{k=0}^n {n \choose k}^4\right\}x^n,
$$
and let $D$ be the differential operator $D=x\,\dfrac{\d}{\d x}$.
Then
\begin{align*}
\sum_{n=0}^\infty \sum_{k=0}^\infty h(n,k)x^{n+k} &=f(x),
\intertext{and}
\sum_{n=0}^\infty \sum_{k=0}^\infty n\,h(n,k)x^{n+k}
&= \frac{1}{5(1+2x)} \left(4xf(x)+3(1+4x)\, Df(x)\right).
\end{align*}
\end{Lemma}

\begin{proof}
On putting $n+k=m$ and applying Lemma~\ref{lemma1}, we have
\begin{align*}
\sum_{n=0}^\infty \sum_{k=0}^\infty h(n,k)x^{n+k}
&= \sum_{m=0}^\infty \left\{\sum_{k=0}^ m h(m-k,k) \right\} x^{m} \\
&= \sum_{m=0}^\infty \left\{\sum_{k=0}^ m {2k \choose k}{2m-2k \choose m-k}{m+k \choose m-k}{m-k \choose k} \right\} x^{m} \\
&= \sum_{m=0}^\infty \left\{\sum_{k=0}^ m {m \choose k}^4 \right\} x^{m} \\
&= f(x).
\end{align*}
To prove the second result of this lemma, start with the satellite identity in Lemma~\ref{lemma2} in the form
\begin{align*}
&
2(1-x)\sum_{n=0}^\infty \sum_{k=0}^\infty n\,h(n,k)x^{n+k}
\\ &\qquad
= 4x\sum_{n=0}^\infty \sum_{k=0}^\infty h(n,k)x^{n+k}
+ 3(1+4x)\sum_{n=0}^\infty \sum_{k=0}^\infty k\,h(n,k)x^{n+k}.
\end{align*}
Add $\displaystyle{3(1+4x)\sum_{n=0}^\infty \sum_{k=0}^\infty n \,h(n,k)x^{n+k}}$ to both sides, and
then apply the first result of this lemma to get
\begin{align*}
&
5(1+2x)\sum_{n=0}^\infty \sum_{k=0}^\infty n\,h(n,k)x^{n+k}
\\ &\qquad
= 4x\sum_{n=0}^\infty \sum_{k=0}^\infty h(n,k)x^{n+k}
+ 3(1+4x)\sum_{n=0}^\infty \sum_{k=0}^\infty (n+k)\,h(n,k)x^{n+k}
\\ &\qquad
= 4x\sum_{n=0}^\infty \sum_{k=0}^\infty h(n,k)x^{n+k}
+ 3(1+4x)\, D \sum_{n=0}^\infty \sum_{k=0}^\infty h(n,k)x^{n+k}
\\ &\qquad
= 4xf(x)+3(1+4x)Df(x).
\end{align*}
Divide both sides by $5(1+2x)$ to complete the proof.
\end{proof}

\begin{Theorem}
The identities \textup{(5.1)--(5.8)} in Sun's Conjecture~\textup{5} in~\cite{sun} are equivalent to the
eight series for $1/\pi$ in Theorem~\textup{5.3} of \cite{cooper10}.
\end{Theorem}

\begin{proof}
From Lemma \ref{fDf}, we deduce that
$$
\sum_{n=0}^\infty \sum_{k=0}^\infty (an+b)\,h(n,k)x^{n+k}
=  \sum_{n=0}^\infty \left\{\sum_{k=0}^n {n \choose k}^4\right\}(An+B)x^n,
$$
where
$$
A=\frac{3a(1+4x)}{5(1+2x)}\quad\text{and}\quad B=\frac{4ax}{5(1+2x)}+b.
$$
For example, taking $(a,b,x)=(95,13,1/36)$ gives
\begin{equation}
\label{y10}
\sum_{n=0}^\infty \sum_{k=0}^\infty (95n+13)\,h(n,k)\, \frac{1}{36^{n+k}}
= 60 \sum_{n=0}^\infty\left\{\sum_{k=0}^n {n \choose k}^4\right\} \left(n+\frac14\right) \frac{1}{36^n}.
\end{equation}
The series on the left occurs in \cite[Conjecture (5.3)]{sun}, whereas
the series on the right is due to Y.~Yang and its value is known to be (e.g., see~\cite[Eq.~(2.2)]{ctyz})
$$
60 \times \frac{3\sqrt{15}}{10\pi} = \frac{18\sqrt{15}}{\pi}.
$$
This proves Conjecture (5.3) in~\cite{sun}.

The series on the right-hand side of \eqref{y10} corresponds to the data associated with $y_A=1/36$
in \cite[Table~1]{cooper10}. In fact, the arguments of $s_k(x)$ in each of Conjectures~(5.1)--(5.8)
are in one to one correspondence with the eight
values\footnote{The entry $4/196$ in \cite[Table~1]{cooper10} is a misprint and should be $1/196$.}
 of $1/y_A$ in~\cite[Table~1]{cooper10}. In the case of Conjecture~(5.1), the series
 corresponding to $y_A=-1/9$ in~\cite[Table~1]{cooper10} diverges.
This can be handled by using the value of $y_C$ in that table and the
associated convergent series given by~\cite[Eq.~(63)]{cooper10}.
This accounts for all of the Conjectures (5.1)--(5.8) in~\cite{sun}.
\end{proof}

\begin{Remark}
Conjectures (5.2)--(5.8) in~\cite{sun} were first proved in the second named author's PhD dissertation \cite[Sections 12.3.4, 12.4.1 and 12.4.2]{wanphd},
using the techniques outlined here.
\end{Remark}

\section{Conjectures (3.1)--(3.10): Level $24$}
\label{sec3}

The Conjectures (3.1)--(3.10) in~\cite{sun} are based on series of the form
\begin{equation}
\label{fx}
\sum_{m=0}^\infty \sum_{k=0}^m {m+k \choose 2k}{2k \choose k}^2 {2m-2k \choose m-k} x^{m+k}
 = \sum_{n=0}^\infty t(n)x^n,
\end{equation}
where
\begin{equation}
\label{tn}
t(n) = \sum_{k=0}^{\lfloor n/2 \rfloor} {n \choose 2k}{2k \choose k}^2 {2n-4k \choose n-2k}.
\end{equation}
Zeilberger's algorithm can be used to show that the sequence $\left\{t(n)\right\}$ satisfies the four-term recurrence relation
\begin{align}
\label{4term}
(n+1)^3t(n+1)&=2(2n+1)(2n^2+2n+1)t(n)\\
&\quad +4n(4n^2+1)t(n-1)-64n(n-1)(2n-1)t(n-2). \nonumber
\end{align}
The single initial condition $t(0)=1$ is enough to start the sequence.

There is a modular parameterization of the series $\left\{t(n)\right\}$. To state it, we will need Dedekind's eta-function
$\eta(\tau)$ and the weight two Eisenstein series $P(q)$; they are defined by
$$
\eta(\tau) = q^{1/24}\prod_{j=1}^\infty (1-q^j), \quad\text{where}\quad q=\exp(2\pi i\tau) \quad\text{and}\quad \Im\tau>0,
$$
and
$$
P(q) = 24\,q\,\frac{\d}{\d q} \log \eta(\tau) = 1-24\sum_{j=1}^\infty \frac{jq^j}{1-q^j}.
$$
\begin{Theorem}
\label{th}
Let
$$
z=\frac14\left(6P(q^{12})-3P(q^6)+2P(q^4)-P(q^2)\right)
+2\eta(2\tau)\eta(4\tau)\eta(6\tau)\eta(12\tau)
$$
and
$$
x= \frac{\eta(2\tau)\eta(4\tau)\eta(6\tau)\eta(12\tau)}{z}.
$$
Let $\left\{t(n)\right\}$ be the sequence defined by equation~\eqref{tn}.
Then in a neighborhood of~$x=0$,
\begin{equation}
\label{zx}
z=\sum_{n=0}^\infty t(n)x^n.
\end{equation}
\end{Theorem}

\begin{proof}
Consider the level~$6$ functions $Z$ and $X$ defined by
$$
Z=\frac14\left(6P(q^6)-3P(q^3)+2P(q^2)-P(q)\right)
$$
and
$$
X=\left(\frac{\eta(\tau)\eta(2\tau)\eta(3\tau)\eta(6\tau)}{Z}\right)^2.
$$
It is known, e.g. \cite[Theorem~3.1]{ctyz}, that in a neighborhood of $X=0$,
$$
Z=\sum_{n=0}^\infty {2n \choose n} u(n) X^n,
$$
where the coefficients $u(n)$ are given by the formula
$$
u(n)=\sum_{j=0}^n {n \choose j}^2 {2j \choose j},
$$
or equivalently by the three-term recurrence relation
$$
(n+1)^2u(n+1)=(10n^2+10n+3)u(n)-9n^2u(n-1)
$$
and initial condition $u(0)=1$.
It follows from the recurrence relation that~$Z$ satisfies a third order linear differential equation
\begin{multline}
\label{DE}
X^2(1-4X)(1-36X)\frac{\d^3Z}{\d X^3} + 3X(1-60X+288X^2)\frac{\d^2Z}{\d X^2}
\\
+(1-132X+972X^2)\frac{\d Z}{\d X}=6(1-18X)Z.
\end{multline}
By using the definitions of $z$, $x$, $Z$ and $X$ given above, it is routine to check that
\begin{equation}
\label{cov}
x=\frac{\sqrt{X}}{1+2\sqrt{X}}\biggr|_{q\rightarrow q^2}
\quad\text{and}\quad
z=\bigl(1+2\sqrt{X}\bigr)Z\Bigr|_{q\rightarrow q^2}.
\end{equation}
On making this change of variables in the differential equation \eqref{DE}, we find that
\begin{multline}
\label{de24}
x^2(1+4x)(1-4x)(1-8x)\frac{\d^3z}{\d x^3} + 3x(1-12x-32x^2+320x^3)\frac{\d^2z}{\d x^2}
\\
+(1-28x-116x^2+1536x^3)\frac{\d z}{\d x} = 2(1+10x-192x^2)z.
\end{multline}
Substitution of the series expansion \eqref{zx} into this differential equation produces the
recurrence relation~\eqref{4term}.
\end{proof}

We also have the following differentiation formula.

\begin{Theorem}
\label{th2}
Let $x$ and $z$ be as for Theorem~\textup{\ref{th}}. Then
\begin{equation}
\label{qdxdq}
q\frac{\d x}{\d q} = z\,x\,\sqrt{(1+4x)(1-4x)(1-8x)}.
\end{equation}
\end{Theorem}
\begin{proof}
With $X$ and $Z$ as in the proof of Theorem~\ref{th}, it is known, e.g. \cite[Section~5.2]{ctyz}, that
$$
q\frac{\d X}{\d q} = Z\,X\,\sqrt{(1-4X)(1-36X)}.
$$
The required formula follows by the change of variables given by~\eqref{cov}.
\end{proof}

The differential equation~\eqref{de24} and the differentiation formula~\eqref{qdxdq} were obtained
independently using a different method by D. Ye~\cite{lawrence}.

Theorems~\ref{th} and~\ref{th2} can be used in a theorem of H.~H.~Chan, S.~H.~Chan and Z.-G.~Liu
\cite[Theorem~2.1]{chan} to produce a family of series for $1/\pi$ of the form
\begin{equation}
\label{pi24}
\frac{1}{2\pi} \times \sqrt\frac{24}{N}
= \sqrt{(1+4x_N)(1-4x_N)(1-8x_N)}\, \sum_{n=0}^\infty \left(n+\lambda_N\right) \, t(n)\, x_N^n,
\end{equation}
where $N$ is a positive integer and
$$
x_N = x\left(\pm e^{-2\pi \sqrt{N/24}}\right).
$$
The formula for $\lambda_N$ is given in~\cite{chan} but it is more complicated so we do not
reproduce it here.
In practice, since $\lambda_N$ is an algebraic number, its value can be recovered symbolically
by computing a sufficiently precise approximation.
A list of values for which $x_N$ is rational, together with
the corresponding values of $N$ and $\lambda_N$, is given in Table~\ref{table1}.
The obvious symmetry in the table between $x(q)$ and $x(-q)$ is explained by the identity
$$
\frac{1}{x(q)} + \frac{1}{x(-q)}=4,
$$
which is a trivial consequence of the definition of $x(q)$ and properties of even and odd functions.

\renewcommand{\arraystretch}{1.5}
\begin{table}
\caption{Data to accompany the series \eqref{pi24}}\label{table1}
\begin{tabular}{|c||c|c||c|c|}
\hline
\multirow{2}{*}{$N$} & \multicolumn{2}{|c||}{$q=\exp(-2\pi\sqrt{N/24})$}
& \multicolumn{2}{|c|}{$q=-\exp(-2\pi\sqrt{N/24}) $}  \\
\cline{2-5}
& $x_N$ & $\lambda_N$ & $x_N$ & $\lambda_N$ \\
\hline
\hline
$1$ & $1/8$ & does not converge & $-1/4$ & does not converge \\ \hline
$3$ & $1/12$ & $1/4$ & $-1/8$ & $1/2$ \\ \hline
$5$ & $1/20$ & $1/4$ & $-1/16$ & $2/5$ \\ \hline
$7$ & $1/32$ & $5/21$ & $-1/28$ & $1/3$ \\ \hline
$13$ & $1/104$ & $1/5$ & $-1/100$ & $3/13$ \\ \hline
$17$ & $1/200$ & $143/238$ & $-1/196$ & $67/340$ \\ \hline
\end{tabular}
\end{table}

The values in Table~\ref{table1} appear to be the only positive integers $N$ that give rise to rational
values of~$x$. Other algebraic values can be determined, e.g.,
$N=11$ and $q=\exp(-2\pi\sqrt{11/24})$ gives $x_{11} = 1/(38+6\sqrt{33})$ and $\lambda_{11}=58/(165+19\sqrt{33})$.
The values in the table corresponding to $N=1$ give rise to divergent series and are not part of the
conjectures.

Conjectures (3.1)--(3.10) in~\cite{sun} can be explained by the values corresponding to
$N=3$, $5$, $7$, $13$ and $17$ in Table~\ref{table1} and the series~\eqref{pi24}.
To complete the proof of these conjectures, we require the satellite identity
\begin{equation} \label{satelliteg}
\sum_{m=0}^\infty\sum_{k=0}^m {m+k \choose 2k}{2k \choose k}^2 {2m-2k \choose m-k} x^{m+k}
\left[x+k(1+x)+m\left(x-\frac12\right)\right]=0
\end{equation}
that holds in a neighborhood of $x=0$,
to produce an analogue of Lemma~\ref{fDf}. We omit the details, as they are similar.

\begin{Remark}
\label{sat-rem}
We note that satellite identities such \eqref{satelliteg} may in fact be first guessed, then proved using multiple Wilf--Zeilberger.
The idea is to assume that the function inside the square brackets takes the form
\[
(a_0+a_1 x+a_2 x^2+ \cdots) + k(b_0+b_1x+b_2x^2+\cdots)+m(c_0+c_1 x+c_2x^2+\cdots),
\]
where $a_i, b_i, c_i$ are undetermined rational coefficients (for $i$ less than some chosen $M$).
The coefficients $a_i, b_i, c_i$ can be found by expanding in powers of $x$ and equating coefficients to obtain
a linear system. Alternatively,
by replacing $x$ with a sufficiently small irrational number,  evaluating the sum to high precision and equating it to 0,
$a_i, b_i, c_i$ may then be determined using an integer relations algorithm, such as PSLQ. See \cite[Section 12.4.2]{wanphd}.
\end{Remark}

\section{Conjectures (2.4)--(2.9): Level 4}
\label{sec4}

Conjectures (2.4)--(2.9) in~\cite{sun} are based on series of the form
$$
 \sum_{n=0}^\infty {2n \choose n} \,z^n\,
\sum_{k=0}^\infty {n \choose k}{2k \choose k} {2n-2k \choose n-k} x^k.
$$
The numerical data given in~\cite{sun} suggest that
$z=x/(1+4x)^2$ for Conjectures (2.4), (2.7) and (2.8), and $z=-x/(1-8x)$
for Conjectures (2.5), (2.6) and (2.9).
Expanding in powers of $x$ leads to:

\begin{Theorem}
\label{4}
The following identities hold in a neighborhood of $x=0$:
\begin{align*}
\lefteqn{\pFq{3}{2}{\frac12,\frac12,\frac12}{1,1}{64x^2}
= \sum_{n=0}^\infty {2n \choose n}^3 x^{2n}} \\
&= \sum_{n=0}^\infty {2n \choose n} \frac{x^n}{(1+4x)^{2n+1}}
\sum_{k=0}^n {n \choose k}{2k \choose k} {2n-2k \choose n-k} x^k \\
&=\sum_{n=0}^\infty {2n \choose n} \frac{(-1)^nx^n}{(1-8x)^{n+1/2}}
\sum_{k=0}^n {n \choose k}{2k \choose k} {2n-2k \choose n-k} x^k.
\end{align*}
\end{Theorem}
\noindent
The first equality in Theorem~\ref{4} is trivial. The substance of the theorem is in the other equalities,
which first appeared in \cite[Theorem 12.3]{wanphd}, and proved using the multiple Wilf--Zeilberger algorithm.

The corresponding satellite identities can be determined, and these give rise to:

\begin{Theorem}
The following identities hold in a neighborhood of $x=0$:
\begin{align}
\sum_{n=0}^\infty {2n \choose n}\,  (an+b)\, & \frac{x^n}{(1+4x)^{2n}}\,
\sum_{k=0}^n {n \choose k}{2k \choose k}{2n -2k \choose n-k}x^k  \label{i1} \\
&=\sum_{n=0}^\infty {2n \choose n}^3 (An+B)x^{2n} \nonumber
\intertext{and}
\sum_{n=0}^\infty {2n \choose n}\,  (cn+d)\, & \frac{(-x)^n}{(1-8x)^{n}}\,
\sum_{k=0}^n {n \choose k}{2k \choose k}{2n -2k \choose n-k}x^k  \label{i11} \\
&=\sum_{n=0}^\infty {2n \choose n}^3 (Cn+D)x^{2n}, \nonumber
\end{align}
where $A$, $B$, $C$ and $D$ are given by
\begin{align*}
A&=\frac{3a}{2}\frac{(1+4x)^2}{1-4x}, & B&=(1+4x)\left(b+\frac{4ax}{1-4x}\right),
\\
C&=\frac{3c}{2}\frac{(1-8x)^{3/2}}{(1-16x^2)} \qquad\text{and} & D&=\sqrt{1-8x}\,\left(d-\frac{4cx(1-2x)}{1-16x^2}\right).
\end{align*}
\end{Theorem}

Conjectures (2.4), (2.7) and (2.8) in~\cite{sun} are obtained by taking
$$
(a,b,x) = \left(12,1,\frac{1}{16}\right),\quad \left(476,103,\frac{-1}{64}\right),\quad \left(140,19,\frac{1}{64}\right)
$$
respectively, in \eqref{i1}. The first set of parameter values produces a
constant multiple of the series
\begin{equation}
\label{r28}
 \sum_{n=0}^\infty {2n \choose n}^3 \left(n+\frac16\right) \frac{1}{256^n} = \frac{2}{3\pi},
\end{equation}
which is originally due to Ramanujan  \cite[Eq.~(28)]{ramanujan_pi}.
The other two sets of parameter values give constant multiples of the series
\begin{equation}
\label{r29}
 \sum_{n=0}^\infty {2n \choose n}^3 \left(n+\frac{5}{42}\right) \frac{1}{4096^n} = \frac{8}{21\pi}
\end{equation}
which is also due to Ramanujan  \cite[Eq.~(29)]{ramanujan_pi}.
The series \eqref{r28} and \eqref{r29} correspond to the values $N=3$ and $N=7$
in \cite[Table~6]{cc}.

In a similar way, Conjectures (2.5), (2.6) and (2.9) in~\cite{sun} are obtained by taking
$$
(c,d,x)=\left(10,1,\frac{-1}{16}\right),\quad
\left(170,37,\frac{1}{64}\right),\quad
\left(1190,163,-\frac{1}{64}\right)
$$
respectively, in \eqref{i11}. The first set of parameter values gives a constant multiple of
Ramanujan's series \eqref{r28}, while the other two sets of values both lead to multiples of~\eqref{r29}.

The parameter values
$$
(a,b,x)=\left(20,7,\frac{-1}{16}\right)\quad\text{and}\quad (c,d,x)=\left(30,11,\frac{1}{16}\right)
$$
also lead to multiples of Ramanujan's series~\eqref{r28}. However, the respective series on the left-hand
sides of \eqref{i1} and \eqref{i11} are divergent, hence they are not listed among the conjectures in~\cite{sun}.

\section{Conjectures (2.1)--(2.3): Level 6}
\label{sec5}

Conjectures (2.1)--(2.3) of \cite{sun} are based on series of the form
$$
\sum_{n=0}^\infty {2n \choose n} z^n
\sum_{k=0}^n {n \choose k}^2 {n+k \choose k} x^k.
$$
Numerical data suggest that $z=x/(1-4x)$ and this leads to:

\begin{Theorem}
\label{domb}
The following identity holds in a neighborhood of $x=0$:
\begin{equation}
\label{div}
\sum_{n=0}^\infty {2n \choose n} \frac{x^n}{(1-4x)^{n+1/2}}
\sum_{k=0}^n {n \choose k}^2 {n+k \choose k} x^k
= \sum_{n=0}^\infty u(n) x^n
\end{equation}
where
$$
(n+1)^3 u(n+1)=(2n+1)(10n^2+10n+4)u(n)-64n^3u(n-1), \quad u(0)=1,
$$
or equivalently,
$$
u(n) = \sum_{j=0}^n{n \choose j}^2{2j \choose j}{2n-2j \choose n-j}.
$$
\end{Theorem}
\noindent
The numbers $u(n)$ are called \textit{Domb numbers}. They are the sequence \texttt{A002895} in
Sloane's database~\cite{sloane}. The series for $1/\pi$ that arise from the Domb numbers
were first studied in \cite{chan}; see also the classification in~\cite[Table~9]{cc}.

Conjectures (2.1), (2.2) and (2.3) in~\cite{sun} involve the values $x=-1/8$, $x=-1/32$ and $x=1/64$, respectively. However,
the series on the right-hand side of~\eqref{div} converges for $|x|<1/16$, so Conjecture (2.1) cannot be
handled by this formula. To obtain a formula that is convergent for all three conjectures, we recall the
identity~\cite[Theorem~3.1]{rogers} that holds in a neighborhood of $x=0$:
\begin{equation}
\label{rogers}
\sum_{n=0}^\infty u(n) x^n
= \sum_{n=0}^\infty {3n \choose n}{2n \choose n}^2 \frac{x^{2n}}{(1-4x)^{3n+1}}.
\end{equation}
The identities \eqref{div} and \eqref{rogers} can be combined and used to produce:

\begin{Theorem}
The following identity holds in a neighborhood of $x=0$:
\begin{multline}
\label{m1}
\sum_{n=0}^\infty {2n \choose n}(an+b) \frac{x^n}{(1-4x)^{n}}
\sum_{k=0}^n {n \choose k}^2 {n+k \choose k} x^k
\\
= \sum_{n=0}^\infty {3n \choose n}{2n \choose n}^2 (An+B) \frac{x^{2n}}{(1-4x)^{3n}},
\end{multline}
where $A$ and $B$ are given by
$$
A=\frac{4a(1+2x)(1-x)}{3(1-4x+8x^2)\sqrt{1-4x}}
\quad\text{and}\quad
B=\frac{1}{\sqrt{1-4x}}\left(\frac{2ax(1-2x)}{1-4x+8x^2}+b\right).
$$
\end{Theorem}

The series on the right-hand side of~\eqref{m1}
converges for $-1/2 < x < 1/16$.  Conjectures (2.1), (2.2) and (2.3)
correspond to the parameter values
$$
(a,b,x)=\left(13,4,\frac{-1}{8}\right),\quad
\left(290,61,\frac{-1}{32}\right)\quad\text{and}\quad
\left(962,137,\frac{1}{64}\right),
$$
respectively. These values produce multiples of the series
\begin{equation}
\label{rr1}
\sum_{n=0}^\infty {3n \choose n}{2n \choose n}^2 \left(n+\frac16\right)\frac{1}{6^{3n}} = \frac{\sqrt{3}}{2\pi},
\end{equation}
\begin{equation}
\label{rr2}
\sum_{n=0}^\infty {3n \choose n}{2n \choose n}^2 \left(n+\frac{2}{15}\right)\frac{1}{2^n\times3^{6n}} =\frac{9}{20\pi}
\end{equation}
and
\begin{equation}
\label{rr3}
\sum_{n=0}^\infty {3n \choose n}{2n \choose n}^2 \left(n+\frac{4}{33}\right)\frac{1}{15^{3n}} =\frac{5\sqrt{3}}{22\pi}.
\end{equation}
The last two of these series are originally due to Ramanujan \cite[Eqs.~(31) and (32)]{ramanujan_pi}
and the other series is due to
J.~M.~Borwein and P.~M.~Borwein \cite[p.~190]{agm}. These series correspond to the
values $N=2$, $4$ and $5$ in~\cite[Table~5]{cc}.
This completes our discussion of Conjectures (2.1)--(2.3) in~\cite{sun}.

\section{Conjectures (2.12)--(2.14), (2.18) and (2.20)--(2.22)}
\label{sec6}

Conjectures (2.10)--(2.28) in~\cite{sun} involve series of the form
$$
\sum_{n=0}^\infty {2n \choose n} \, z^n\,
\sum_{k=0}^n {2k \choose k}^2 {2n-2k \choose n-k}\,x^k.
$$
We present two cases where the data allow $z$ to be identified as a function of~$x$.

\subsection{Conjectures (2.13), (2.18) and (2.22): Level 6}
\label{sec6.1}
The data for these conjectures satisfy the relation $z=-x/(1-16x)$. This leads us to discover:

\begin{Theorem}
The following identity holds in a neighborhood of $x=0$:
\begin{equation}
\label{clever}
\sum_{n=0}^\infty {2n \choose n} \frac{(-x)^n}{(1-16x)^{n+1/2}}
\sum_{k=0}^n {2k \choose k}^2 {2n-2k \choose n-k}x^k
= \sum_{n=0}^\infty u(n) x^n,
\end{equation}
where $\left\{ u(n)\right\}$ are the Domb numbers introduced in Theorem~\textup{\ref{domb}}.
\end{Theorem}

Just as for Theorem~\ref{domb}, the radius of convergence of the series on the right-hand side of~\eqref{clever}
is not large enough to handle all of the Conjectures (2.13), (2.18) and (2.22) in~\cite{sun}. Therefore, we
use~\eqref{clever} in conjunction with~\eqref{rogers} to produce the identity
\begin{multline}
\label{herewegoagain}
\sum_{n=0}^\infty {2n \choose n}(an+b) \frac{(-x)^n}{(1-16x)^{n}}
\sum_{k=0}^n {2k \choose k}^2 {2n-2k \choose n-k}x^k
\\
= \sum_{n=0}^\infty {3n \choose n}{2n \choose n}^2 (An+B) \frac{x^{2n}}{(1-4x)^{3n}},
\end{multline}
where $A$ and $B$ are given by
$$
A=\frac{4a(1-16x)^{3/2}(1+2x)}{3(1-4x)(1-8x)} \quad\text{and}\quad
B=\frac{\sqrt{1-16x}}{1-4x}\left(b-\frac{4ax}{1-8x}\right).
$$
The series on the right-hand side of~\eqref{herewegoagain} converges
for $-1/2 < x < 1/16$.
Conjectures (2.13), (2.18) and (2.22) in~\cite{sun} correspond to the parameter values
$$
(a,b,x)=\left(1,0,\frac{-1}{8}\right),\quad
\left(10,1,\frac{-1}{32}\right)\quad\text{and}\quad
\left(14,3,\frac{1}{64}\right),
$$
respectively. These values produce multiples of the series \eqref{rr1}, \eqref{rr2} and \eqref{rr3}, respectively.
This completes our discussion of Conjectures (2.13), (2.18) and (2.22) in~\cite{sun}.

\subsection{Conjectures (2.12), (2.14), (2.20) and (2.21): Level 6}
\label{sec6.2}
The data for these conjectures satisfy the relation $z=x/(1+4x)^2$. Expanding in powers of~$x$ leads to:

\begin{Theorem}
\label{212}
The following identity holds in a neighborhood of $x=0$:
\begin{equation}
\label{nice}
\sum_{n=0}^\infty {2n \choose n} \frac{x^n}{(1+4x)^{2n+1}}
\sum_{k=0}^n {2k \choose k}^2 {2n-2k \choose n-k}x^k
= \sum_{n=0}^\infty t(n) x^n,
\end{equation}
where the sequence $\left\{t(n)\right\}$ is defined by the four-term recurrence relation
\begin{align}
\label{u4}
(n+1)^3t(n+1)&=-2n(n+1)(2n+1)t(n)+16n(5n^2+1)t(n-1)
\\ &\qquad
-96n(n-1)(2n-1)t(n-2)
\nonumber
\end{align}
and initial condition $t(0)=1.$
The series on the left-hand side of \eqref{nice} converges for $-1/12 < x < 1/4$, while the series on the right-hand side converges for $|x|<1/12$.
\end{Theorem}

In order to gain access to properties of the sequence $\left\{t(n)\right\}$, we recall the
following result of Chan et al.\ \cite[Eq.~(4.13)]{chan}.

\begin{Lemma}
Let $z$ and $y$ be the level $6$ modular forms defined by
\begin{equation}
\label{zy}
z= \prod_{j=1}^\infty \frac{(1-q^j)^4(1-q^{3j})^4}{(1-q^{2j})^2(1-q^{6j})^2}
\quad\text{and}\quad
y = q\prod_{j=1}^\infty \frac{(1-q^{2j})^6(1-q^{6j})^6}{(1-q^{j})^6(1-q^{3j})^6}.
\end{equation}
Let $\left\{u(n)\right\}$ be the Domb numbers, which were defined in Theorem~\textup{\ref{domb}}.
Then in a neighborhood of $y=0$,
\begin{equation}
\label{zy1}
z=\sum_{n=0}^\infty (-1)^n u(n)y^n.
\end{equation}
\end{Lemma}

The next result gives a modular parameterization for the sequence $\left\{t(n)\right\}$. It also provides a connection
with the Domb numbers.

\begin{Theorem}
\label{par}
Let $z$ and $y$ be the level $6$ modular forms defined by \eqref{zy}
and let $\left\{u(n)\right\}$ be the Domb numbers, which were defined in Theorem~\textup{\ref{domb}}.
Let $Z$ and $x$ be defined by
\begin{equation}
\label{ZXdef}
Z=(1+4y)z\qquad\text{and}\qquad x=\frac{y}{1+4y}.
\end{equation}
Then in a neighborhood of $q=0$,
\begin{align}
Z&=\frac12\left(3P(q^6)-P(q^2)\right) \label{d1} \\
&=\sum_{n=0}^\infty t(n)x^n \label{d2} \\
&= \sum_{n=0}^\infty (-1)^nu(n) \frac{x^n}{(1-4x)^{n+1}} \label{d3} \\
&= \sum_{n=0}^\infty {3n \choose n}{2n \choose n}^2 \, x^{2n}(1-4x)^n. \label{d4}
\end{align}
\end{Theorem}

\begin{proof}
By \eqref{zy} and \cite[Eqs.~(32.66) and (33.2)]{fine}, we have
\begin{align*}
Z &= (1+4y)z \\
&= \prod_{j=1}^\infty \frac{(1-q^j)^4(1-q^{3j})^4}{(1-q^{2j})^2(1-q^{6j})^2}
+4q \prod_{j=1}^\infty \frac{(1-q^{2j})^4(1-q^{6j})^4}{(1-q^{j})^2(1-q^{3j})^2} \\
&= \frac16\left(12P(q^6)-3P(q^3)-4P(q^2)+P(q)\right)\\
& \qquad +\frac16\left(-3P(q^6)+3P(q^3)+P(q^2)-P(q)\right) \\
&= \frac12\left(3P(q^6)-P(q^2)\right).
\end{align*}
This proves \eqref{d1}.

Next, Chan et al.\ \cite[Eq.~(4.10)]{chan} showed that $z$ satisfies a third order differential
equation with respect to $y$:
\begin{multline*}
y^2(1+4y)(1+16y)\frac{\d^3z}{\d y^3}+3y(1+30y+128y^2)\frac{\d^2z}{\d y^2}
\\
+(1+168y+448y^2)\frac{\d z}{\d y}+4(1+16y)z=0.
\end{multline*}
On making the change of variables given by \eqref{ZXdef}, we deduce that
\begin{multline*}
x^2(1-4x)^2(1+12x)\frac{\d^3Z}{\d x^3}+3x(1-4x)(1+10x-120x^2)\frac{\d^2Z}{\d x^2}
\\
+(1+12x-576x^2+2304x^3)\frac{\d Z}{\d x}+96x(6x-1)Z=0.
\end{multline*}
On expanding $Z$ in powers of $x$ and substituting into the differential equation, we obtain the
recurrence relation~\eqref{u4}. The proof of~\eqref{d2} may be completed by noting that $Z=1$ and $x=0$ when
$q=0$, therefore $t(0)=1$.

To prove \eqref{d3}, use \eqref{zy1} and \eqref{ZXdef} to get
$$
Z=(1+4y)z=(1+4y)\sum_{n=0}^\infty (-1)^nu(n)y^n = \sum_{n=0}^\infty (-1)^n u(n) \frac{x^n}{(1-4x)^{n+1}}.
$$

Finally, \eqref{d4} can be obtained by applying \eqref{rogers} to \eqref{d3}.
\end{proof}

Combining \eqref{nice} with \eqref{d2} and \eqref{d4} gives the identity
\begin{multline} \label{heun}
\sum_{n=0}^\infty {2n \choose n} \frac{x^n}{(1+4x)^{2n+1}}
\sum_{k=0}^n {2k \choose k}^2 {2n-2k \choose n-k}x^k
\\
= \sum_{n=0}^\infty {3n \choose n}{2n \choose n}^2 \,x^{2n}(1-4x)^n.
\end{multline}
Equation \eqref{heun} can be used to produce:

\begin{Theorem}
The following identity holds in a neighborhood of $x=0$:
\begin{multline}
\label{55}
\sum_{n=0}^\infty (an+b){2n \choose n} \frac{x^n}{(1+4x)^{2n}}
\sum_{k=0}^n {2k \choose k}^2 {2n-2k \choose n-k}x^k
\\
= \sum_{n=0}^\infty \left(An+B\right){3n \choose n}{2n \choose n}^2 \,x^{2n}(1-4x)^n,
\end{multline}
where
$$
A = \frac{4a(1+4x)^2(1-6x)}{3(1-4x)^2}
\qquad\text{and}\qquad
B = (1+4x)\left(\frac{4ax}{1-4x}+b\right).
$$
The series on the right-hand side of \eqref{55} converges for $-1/12 < x < 1/6$.
\end{Theorem}

Conjectures (2.14), (2.20) and (2.21) in~\cite{sun} correspond to the data
$$
(a,b,x)=\left(6,-1,\frac{1}{12}\right),\quad
\left(12,1,\frac{1}{36}\right)\quad\text{and}\quad
\left(24,5,\frac{-1}{60}\right),
$$
respectively. These values lead to multiples of the series \eqref{rr1}, \eqref{rr2} and \eqref{rr3}, respectively.

Conjecture (2.12) in~\cite{sun} corresponds to the values $x=1/6$, in which case the series on the right-hand side
of~\eqref{55} is divergent.
Therefore, we proceed by a different method.
By \cite[Theorem 1]{translation} we have
$$
\lim_{w\to1^-} \sqrt{1-w}\, \sum_{n=0}^\infty {3n \choose n}{2n \choose n}^2\,n\, \left(\frac{w}{108}\right)^n
=\frac{\sqrt3}{2\pi}.
$$
Make the change of variables $w=108x^2(1-4x)$ and
observe that $w \to 1^-$ as $x \rightarrow (1/6)^-$, we get
$$
\frac{\sqrt3}{2\pi}
= \lim_{x\to(1/6)^-} \sqrt{(1-6x)^2(1+12x)}\sum_{n=0}^\infty {3n \choose n}{2n \choose n}^2\,n\,(x^2(1-4x))^n.
$$
Now apply \eqref{55} with $a=1$ and $b=-2$, and note that
$4x/(1-4x)-2$ vanishes at $x=1/6$. This produces
\begin{align*}
\frac{\sqrt3}{2\pi}
&= \lim_{x\to(1/6)^-} \sqrt{1+12x}\cdot\frac{3(1-4x)^2}{4(1+4x)^2}
\sum_{n=0}^\infty {3n \choose n}{2n \choose n}^2\,
\\ &\quad\times
\biggl(\frac{4(1+4x)^2(1-6x)}{3(1-4x)^2}\,n+(1+4x)\biggl(\frac{4x}{1-4x}-2\biggr)\biggr)
(x^2(1-4x))^n
\\
&= \lim_{x\to(1/6)^-} \sqrt{1+12x}\cdot\frac{3(1-4x)^2}{4(1+4x)^2}
\\ & \quad\times
\sum_{n=0}^\infty (n-2){2n \choose n} \frac{x^n}{(1+4x)^{2n}}
\sum_{k=0}^n {2k \choose k}^2 {2n-2k \choose n-k}x^k
\\
&=\frac{3\sqrt3}{100}
\sum_{n=0}^\infty (n-2){2n \choose n} \biggl(\frac3{50}\biggr)^n
\sum_{k=0}^n {2k \choose k}^2 {2n-2k \choose n-k}\biggl(\frac16\biggr)^k.
\end{align*}
This gives us a proof of Conjecture (2.12) in~\cite{sun}.

Equation \eqref{heun} was deduced via a different path in \cite[Theorem 12.4]{wanphd}:
a Heun-type differential equation was obtained for the left-hand side, which was then explicitly solved and the solution transformed into the right-hand side.
Conjecture (2.12) was also proved in \cite{wanphd}, by first applying Clausen's theorem to convert the right-hand side of \eqref{heun} into the square of a $_2F_1$,
followed by evaluating the $_2F_1$'s with Gauss' second summation theorem and one of its contiguous versions.

\medskip
There is also a companion result to Theorem~\ref{par}:

\begin{Theorem}
\label{par1}
Let $z$ and $y$ be the level $6$ modular forms defined by~\eqref{zy}.
Let $Z^*$ and $x^*$ be defined by
$$
Z^*=(1+16y)z\qquad\text{and}\qquad x^*=\frac{y}{1+16y} .
$$
Then in a neighborhood of $q=0$,
\begin{align*}
Z^*&=\frac12\left(3P(q^3)-P(q)\right) \\
&=\sum_{n=0}^\infty v(n)(x^*)^n \\
&= \sum_{n=0}^\infty (-1)^nu(n) \frac{(x^*)^n}{(1-16x^*)^{n+1}} \\
&=\sum_{n=0}^\infty {3n \choose n}{2n \choose n}^2 (x^*)^n(1-16x^*)^{2n},
\end{align*}
where the sequence $\left\{ v(n) \right\}$ satisfies the recurrence relation
\begin{align*}
(n+1)^3v(n+1)&=(2n+1)(22n^2+22n+12)v(n)-128n(5n^2+1)v(n-1)
\\ &\qquad
+1536n(n-1)(2n-1)v(n-2),
\end{align*}
and $\left\{u(n)\right\}$ are the Domb numbers.
\end{Theorem}

It would be interesting to have an analogue of Theorem~\ref{212} that involves the sequence $\left\{v(n)\right\}$.

\section{Conjectures (6.3)--(6.13)}
\label{sec7}

Conjectures (6.3)--(6.13) in~\cite{sun} are based on the generating function
\begin{equation}
\label{gf}
\sum_{n=0}\binom{2n}n z^n\sum_{k=0}^n{\binom{2k}k}^2\binom{k}{n-k}x^k.
\end{equation}
The numerical data for Conjectures (6.3)--(6.7) fit the relation
$$
z=\frac x{(1-x)^2},
$$
while for Conjectures (6.8)--(6.13) we have
$$
z=-\frac1{2(1+4x)}.
$$
We consider each case separately.

\subsection{Conjectures (6.3)--(6.7): Level 14}
\label{sec7.1}
Expanding in powers of $x$ gives
\begin{equation}
\label{e1}
\sum_{n=0}^\infty\binom{2n}n\frac{x^n}{(1-x)^{2n+1}}\sum_{k=0}^n{\binom{2k}k}^2\binom{k}{n-k}x^k
=\sum_{n=0}^\infty a(n)x^n
\end{equation}
where
\begin{align*}
(n+1)^3a(n+1)&=(2n+1)(3n^2+3n+1)a(n)
\\ &\quad
+n(47n^2+4)a(n-1)+14n(n-1)(2n-1)a(n-2)
\end{align*}
and $a(0)=1$.
The series expansion of a function in \cite[Eq.~(5)]{zudilinBAMS} involves the same coefficients, that is,
\begin{equation}
\label{e2}
\sum_{n=0}^\infty\binom{2n}n\frac{x^n}{(1+x)^{2n+1}}\sum_{k=0}^n\binom{n}{k}\binom{n+k}n\binom{2k}kx^k
=\sum_{n=0}^\infty a(n)x^n.
\end{equation}
The identities \eqref{e1} and \eqref{e2} can be used to establish the interesting result
\begin{align} \label{interesting}
&
\sum_{n=0}^\infty \left((1-x)^2n+(\lambda-1)\right)
\binom{2n}n\frac{x^n}{(1+x)^{2n+1}}\sum_{k=0}^n\binom{n}{k}\binom{n+k}n\binom{2k}kx^k
\\ &\quad
=\sum_{n=0}^\infty\left((1+x)^2n+(\lambda+1)\right)
\binom{2n}n\frac{x^n}{(1-x)^{2n+1}}\sum_{k=0}^n{\binom{2k}k}^2\binom{k}{n-k}x^k \nonumber
\end{align}
which holds for any constant $\lambda$.
This in turn can be used with the results in~\cite{zudilinBAMS} to show that
Conjectures (6.3)--(6.7) are equivalent to Conjectures (VII5), (VII1), (VII3), (VII4) and (VII6) in \cite{sun},
respectively. To see the correspondence, compare the values of $x$ in
\cite[Table~1]{zudilinBAMS} with the arguments of $P_k$ in \cite[Eqs.~(6.3)--(6.7)]{sun}.
Since Conjectures (VII1) and (VII3)--(VII6) have been proved in~\cite{zudilinBAMS}, the truth of
Conjectures (6.3)--(6.7) follows from~\eqref{interesting}.

Before continuing to the next set of conjectures, we offer the following additional comments about
the sequence $\left\{a(n)\right\}$. Equating coefficients of~$x^n$ in~\eqref{e1} and~\eqref{e2}
leads to the following formulas for $a(n)$ as sums of binomial coefficients, respectively:
\begin{align}
\label{14binomial}
a(n)
&= \sum_{j,k} {n+j \choose 2j+2k} { 2j+2k \choose j+k} {2k \choose k}^2 {k \choose j} \\
&= \sum_{j,k} (-1)^{n-j}{n+j \choose 2j+2k} {2j+2k \choose j+k} {2k \choose k} {j+2k \choose k} {j+k \choose k}. \nonumber
\end{align}
It can be shown that
\begin{equation}
\label{14}
\sum_{n=0}^\infty a(n)\left(\frac{x}{1+5x+8x^2}\right)^{n+1}
 =\sum_{n=0}^\infty A(n)\left(\frac{x}{1+9x+8x^2}\right)^{n+1},
\end{equation}
where
\begin{align*}
(n+1)^3A(n+1)&=(2n+1)(11n^2+11n+5)A(n)
\\ &\quad
-n(121n^2+20)A(n-1)+98n(n-1)(2n-1)A(n-2)
\end{align*}
and $A(0)=1$. The sequence $\left\{A(n)\right\}$ was first studied in \cite[Example~6]{translation}.
It was shown in~\cite{level14} that the sequence $\left\{A(n)\right\}$ can be parameterized by
level~$14$ modular forms. The modular parameterization for $\left\{a(n)\right\}$
is inherited from this by \eqref{14}.
Both $\left\{a(n)\right\}$ and $\left\{A(n)\right\}$ possess many remarkable arithmetic properties
that are beyond the scope of this work; we plan to discuss them in a forthcoming project in detail.

\subsection{Conjectures (6.8)--(6.13): Level 2}
\label{sec7.2}
The data for Conjectures (6.8)--(6.13) in~\cite{sun} suggest that $z$ and $x$ in the generating function~\eqref{gf}
are related by $z=-1/(2(1+4x))$. We replace $x$ with $4x$ throughout, and consider the function
$$
g(x) = \sum_{n=0}^\infty {2n \choose n} \frac{(-1)^n}{2^n(1+16x)^{n+1/2}}
\sum_{k=0}^n {2k \choose k}^2{k \choose n-k}\,(4x)^k.
$$
The series $g$ can be seen to converge in a neighborhood of $x=0$ by
noting that the non-zero terms in the inner sum occur only when
$\lceil n/2 \rceil \leq k \leq n$, and so the series may be written in the form
$$
g(x) = \sum_{n=0}^\infty\binom{2n}n\frac{(-1)^n(4x)^{\lceil n/2\rceil}}{2^n(1+16x)^{n+1/2}}
\sum_{k=\lceil n/2\rceil}^n{\binom{2k}k}^2\binom{k}{n-k}(4x)^{k-\lceil n/2\rceil}.
$$
Expanding in powers of $x$ gives
$$
g(x) = \sum_{n=0}^\infty {4n \choose 2n} {2n \choose n}^2 x^{2n}.
$$
This can be used to produce the identity
\begin{multline}
\label{level2}
\sum_{n=0}^\infty {2n \choose n} (an+b) \frac{(-1)^n}{2^n(1+16x)^{n}}
\sum_{k=0}^n {2k \choose k}^2{k \choose n-k}\,(4x)^k
\\
= \sum_{n=0}^\infty {4n \choose 2n} {2n \choose n}^2 (An+B)x^{2n},
\end{multline}
where
$$
A=\frac{4a(1+16x)^{3/2}}{1-48x}
\quad\text{and}\quad
B=(1+16x)^{1/2}\left(b+\frac{32ax}{1-48x}\right).
$$
Conjectures (6.8)--(6.13) in~\cite{sun} correspond to the data\footnote{Multiply the argument of $P_k$ in
each of the Conjectures (6.8)--(6.13) in \cite{sun} by~$4$, and then take the reciprocal to get the values of $x$
in the data.}
\begin{align*}
(a,b,x)&=\left(130,41,\frac{-1}{784}\right),\;
\left(46,13,\frac{1}{784}\right),
\\&\quad
\left(510,143,\frac{-1}{1584}\right),\;
\left(42,11,\frac{1}{1584}\right)
\intertext{and}
&\quad
\left(1848054,309217,\frac{-1}{396^2}\right),\;
\left(171465,28643,\frac{1}{396^2}\right),
\end{align*}
respectively. If either of the two data sets corresponding to $\pm 1/784$ are inserted in~\eqref{level2},
the results are multiples of the series
$$
\sum_{n=0}^\infty {4n \choose 2n}{2n \choose n}^2 \left(n+\frac{3}{40}\right) \frac{1}{28^{4n}}=\frac{49\sqrt{3}}{360\pi}.
$$
Similarly, the data corresponding to $\pm 1/1584$ produce multiples of the series
$$
\sum_{n=0}^\infty {4n \choose 2n}{2n \choose n}^2 \left(n+\frac{19}{280}\right) \frac{1}{1584^{2n}}=\frac{9\sqrt{11}}{140\pi},
$$
while the data corresponding to $\pm 1/396^2$ lead to
$$
\sum_{n=0}^\infty {4n \choose 2n}{2n \choose n}^2 \left(n+\frac{1103}{26390}\right) \frac{1}{396^{4n}}=\frac{9801\sqrt{2}}{105560\pi}.
$$
These are Ramanujan's series~\cite[Eqs.~(42)--(44)]{ramanujan_pi}. They correspond to the
values $N=9$, $11$, $29$ and $q>0$ in \cite[Table 4]{cc}.

\section{Further examples: the \$520 series}
\label{sec8}

We mention one further set of examples for which the techniques of this paper can be used.
The following identity holds in a neighborhood of $x=0$:
\begin{multline}
\label{i10}
\sum_{n=0}^\infty {2n \choose n}(an+b)\frac{x^n}{(1+2x)^{2n}}\sum_{k=0}^n {n \choose k}^2 {2n-2k \choose n-k}x^k
\\
= \sum_{n=0}^\infty (An+B) \left\{\sum_{k=0}^\infty {n \choose k}^4\right\} x^n,
\end{multline}
where
$$
A=\frac{4a(1-x)(1+2x)^2}{5(1-4x)}
\quad\text{and}\quad
B=(1+2x)\left(b+\frac{6ax(2-x)}{5(1-4x)}\right).
$$

\renewcommand{\arraystretch}{2.3}
\begin{table}
\caption{Specialization of the two-variable special series}\label{table2}
\begin{tabular}{|c||c|c|c|}
\hline
Underlying series & Specialization & Reference & Level \\
\hline
\multirow{2}{*}{$\displaystyle{\sum_n{2n \choose n} z^n \sum_k {n \choose k}{2k \choose k}{2n-2k \choose n-k}x^k}$ }
&{$\displaystyle{z=\frac{x}{(1+4x)^2}}$}&Eq.~\eqref{i1}&$4$ \\\cline{2-4}
&{$\displaystyle{z=\frac{-x}{1-8x}}$}&Eq.~\eqref{i11}&$4$\\
\hline
\multirow{2}{*}{$\displaystyle{\sum_n{2n \choose n} z^n \sum_k {n \choose k}^2{n+k \choose k}x^k}$ }
&{$\displaystyle{z=\frac{x}{1-4x}}$}&Eq.~\eqref{div}&$6$ \\ \cline{2-4}
&$\displaystyle{x=\frac{1}{t+1},\;z=t^2}$& \cite{zudilinBAMS} &3 \\
\hline
\multirow{2}{*}{$\displaystyle{\sum_n{2n \choose n} z^n \sum_k {2k \choose k}^2{2n-2k \choose n-k}x^k}$ }
&{$\displaystyle{z=\frac{x}{(1+4x)^2}}$}&Eq.~\eqref{nice}&$6$ \\\cline{2-4}
&{$\displaystyle{z=\frac{-x}{1-16x}}$}&Eq.~\eqref{clever}&$6$\\
\hline
\multirow{2}{*}{$\displaystyle{\sum_n{2n \choose n} z^n \sum_k {2k \choose k}^2{k \choose n-k}x^k}$ }
&{$\displaystyle{z=\frac{x}{(1-x)^2}}$}&Eq.~\eqref{e1}&$14$ \\\cline{2-4}
&{$\displaystyle{z=\frac{-1}{2(1+4x)}}$}&Eq.~\eqref{level2}&$2$\\
\hline
$\displaystyle{\sum_n{2n \choose n} z^n \sum_k {n \choose k}^2{2n-2k \choose n-k}x^k}$
&{$\displaystyle{z=\frac{x}{(1+2x)^2}}$}&Eq.~\eqref{i10}&$10$ \\
\hline
\multirow{2}{*}{$\displaystyle{\sum_n{2n \choose n} z^n \sum_k {n \choose k}^2{2n-2k \choose n}x^{k}}$}
&{$\displaystyle{x=t^2,\;z=\frac{t}{(1+3t)^2}}$}&Eq.~\eqref{dash}&$14$ \\ \cline{2-4}
&{$\displaystyle{z=\frac{x}{1+4x}}$}&Eq.~\eqref{new}&$10$\\
\hline
$\displaystyle{\sum_n{2n \choose n} z^n \sum_k {n \choose k}{n+k \choose n}{2k \choose k}x^k}$
&{$\displaystyle{z=\frac{x}{(1+x)^2}}$}&\cite[Eq.~(5)]{zudilinBAMS}&$7$ \\
\hline
\end{tabular}
\end{table}

Taking $a=5440$, $b=1201$ and $x=-1/64$ gives
\begin{multline}
\label{520}
\sum_{n=0}^\infty {2n \choose n}(5440n+1201)\left(\frac{-4}{31}\right)^{2n}
\sum_{k=0}^n {n \choose k}^2 {2n-2k \choose n-k}
\left(\frac{-1}{64}\right)^k
\\
= \frac{62465}{16} \sum_{n=0}^\infty \left(n+\frac{1}{4}\right)
\left\{\sum_{k=0}^n {n \choose k}^4 \right\}\left(\frac{-1}{16}\right)^n.
\end{multline}
The series on the left-hand side is \cite[Eq.~(3.24$'$)]{sun}, which is equivalent to another identity
\cite[Eq.~(3.24)]{sun} for which a \$520 prize was offered to the first correct solution. That solution
was given by Rogers and Straub~\cite{rogersstraub}.
The value of the series on the right is given by~\cite[Theorem 5.3, $N=9$]{cooper10}.
Hence, we obtain another proof of the `\$520 challenge' series.

The identity~\eqref{i10} also provides
alternative proofs of (3.28), (3.11$'$), (3.13$'$), (3.15$'$), (3.17$'$) and (3.25$'$)
in~\cite{sun}, that were proved by Rogers and Straub~\cite{rogersstraub}.
We note here that proofs for (3.11$'$), (3.13$'$), (3.15$'$), as well as (3.16$'$), (3.18$'$), (3.19$'$), were given in~\cite{wanphd}.

The identities (3.12$'$), (3.14$'$) and (3.18$'$) are equivalent to (3.12), (3.14) and (3.18). They can be handled using
\begin{equation}
\label{dash}
\sum_{n=0}^\infty {2n \choose n} \frac{x^n}{(1+3x)^{2n+1}}\sum_{k=0}^\infty {n \choose k}^2 {2n-2k \choose n}x^{2k}
= \sum_{n=0}^\infty a(n)x^n
\end{equation}
where $a(n)$ is the level $14$ sequence that appears in~\eqref{e1} and~\eqref{e2}. Equating coefficients
of~$x^n$ gives yet another formula for $a(n)$ as a sum of binomial coefficients, to go along with~\eqref{14binomial},
namely
$$
a(n)= \sum_{j,k} {n+j-k \choose 2j+2k} {2j+2k \choose j+k} {j+k \choose k}^2 {2j \choose j+k} (-3)^{n-j-3k}.
$$

\section{Summary and afterthoughts}
\label{sec9}

Table~\ref{table2} summarizes the specializations of the two-variable special series used in this work:
the resulting single-variable series are solutions of third-order linear differential equations, for which
formulas for $1/\pi$ are already established in the literature.
Most of the entries in Table~\ref{table2} were originally guessed on the basis of Sun's conjectural identities in~\cite{sun},
but we also performed an independent computer investigation to search for other linear and quadratic specializations
of the underlying series. The only additional series produced by the search is
\begin{equation}
\label{new}
\sum_{n=0}^\infty {2n \choose n}\frac{x^n}{(1+4x)^{n+1/2}}\sum_{k=0}^\infty {n \choose k}^2{2n-2k \choose n}x^k
= \sum_{n=0}^\infty \left\{\sum_{k=0}^n {n \choose k}^4 \right\}x^n.
\end{equation}
One of the big surprises of our project is a solid presence, in modular parameterizations of third-order linear differential equations,
of level 14 modular forms and functions
(cf.\ Sections~\ref{sec7} and~\ref{sec8}); at the same time the similarly exotic level 15~\cite{level14} does not show up at all.

We note that the non-specialized generating functions from \cite{sun} are expected to be representable as products
of two power series, each satisfying a second-order equation; e.g.~see the final Question in \cite{zudilinBAMS}.
This expectation is shown to be true in many cases and it is the driving force behind the universal methods of establishing Sun's conjectures and similar identities in
\cite{chanwanzudilin,guillera,rogersstraub,wan,wanzudilin}. Two further examples of such factorizations follow from the two-variable identities
\begin{align}
\label{trans1}
&
\sum_n\biggl(\frac zy\biggr)^n\sum_k\binom{n-k}k\binom{2k}k{\binom{2n-2k}{n-k}}^2\biggl(\frac{y(y^2-1)}{4z}\biggr)^k
\\ &\qquad
=\sum_{m=0}^\infty{\binom{2m}m}^2P_m\biggl(\frac{y^2+1}{2y}\biggr)z^m,
\nonumber\\
\label{trans2}
&
\sum_n\binom{2n}n\biggl(\frac{y^2-1}{4y^2}\biggr)^n\sum_k{\binom nk}^2\binom{2k}k\biggl(\frac{4z}{y^2-1}\biggr)^k
\\ &\qquad
=y\sum_{m=0}^\infty{\binom{2m}m}^2P_{2m}(y)z^m
\nonumber
\end{align}
and the corresponding factorizations \cite{chanwanzudilin,wanzudilin} of generating functions of Legendre polynomials
$$
P_n(x)
=\sum_{k=0}^n\binom nk\binom{n+k}k\biggl(\frac{x-1}2\biggr)^k.
$$
The former transformation \eqref{trans1} allows one to deal with \cite[Conjecture~(3.29)]{sun}
(namely, by making it equivalent to \cite[Eq.~(I3)]{sun} established in \cite{chanwanzudilin}),
while the latter one \eqref{trans2} paves the ground for proving the family of conjectures (3.N$'$) on Sun's list~\cite{sun}
in exactly the same way as in~\cite{rogersstraub}.

A drawback of using such two-variable factorizations in the proofs of the formulas for $1/\pi$ is the relatively cumbersome analysis: compare our proof
of Sun's Conjecture~(3.24$'$) from Section~\ref{sec8} with the proof of his (equivalent) Conjecture~(3.24) given in~\cite{rogersstraub}.
An advantage is that transformations are also available for the two-variable series. One such example,
$$
\sum_{k=0}^n\frac{{\binom{2k}k}^2{\binom{2n-2k}{n-k}}^2}{\binom nk}\,y^k
=\biggl(-\frac{16y}{1+y}\biggr)^n\sum_{k=0}^n\binom k{n-k}{\binom{2k}k}^2\biggl(-\frac{(1+y)^2}{16y}\biggr)^k,
$$
follows from the classical Whipple's quadratic transformation and reduces the verification of Sun's \cite[Conjecture (6.14)]{sun}
to one related to the generating function
$$
\sum_{n=0}^\infty\binom{2n}nz^n\sum_{k=0}^n{\binom{2k}k}^2\binom k{n-k}x^k
$$
considered in Section~\ref{sec7}. Unfortunately, the corresponding values $x=-9/20$ and $z=-1/216$ or $z=-5/216$ (depending on whether $y=1/5$ or $y=5$)
do not match the patterns we have discovered.

\bigskip
\noindent
{\bf Acknowledgement.} We thank the referee for helpful comments and suggestions.


\begin{thebibliography}{99}

\bibitem{aar}

G.~E.~Andrews, R.~Askey and R.~Roy,
{\em Special functions.}
Cambridge University Press, 1999.

\bibitem{agm}
J.~M.~Borwein and P.~B.~Borwein,
{\em Pi and the AGM. A study in analytic number theory and computational complexity.}
Wiley, New York, 1987.

\bibitem{chan}
H.~H.~Chan, S.~H.~Chan and Z.-G.~Liu,
{\em Domb's numbers and Ramanujan-Sato type series for $1/\pi$,}
Advances in Mathematics, {\bf 186} (2004) 396--410.

\bibitem{cc}
H.~H.~Chan and S.~Cooper,
{\em Rational analogues of Ramanujan's series for $1/\pi$,}
Math. Proc. Cambridge Phil. Soc., {\bf 153} (2012), 361--383.

\bibitem{ctyz}
H.~H.~Chan, Y.~Tanigawa, Y.~Yang and W.~Zudilin,
{\em New analogues of Clausen's identities arising from the theory of modular forms,}
Adv. Math., {\bf 228} (2011), 1294--1314.

\bibitem{chanwanzudilin}
H.~H.~Chan, J.~G.~Wan and W.~Zudilin,
{\em Legendre polynomials and Ramanujan-type series for $1/\pi$,}
Israel J. Math., {\bf 194} (2013), 183--207.

\bibitem{mathematika}
H.~H.~Chan and W.~Zudilin,
{\em New representations for Ap{\'e}ry-like sequences,}
Mathematika, {\bf 56} (2010), 107--117.

\bibitem{cooper10}
S.~Cooper,
{\em Level $10$ analogues of Ramanujan's series for $1/\pi$,}
J. Ramanujan Math. Soc., {\bf 27} (2012), 59--76.

\bibitem{level14}
S.~Cooper and D.~Ye,
{\em Level $14$ and $15$ analogues of Ramanujan's elliptic functions to alternative bases,}
Trans. Amer. Math. Soc., {\em to appear}.

\bibitem{fine}
N.~Fine,
{\em Basic Hypergeometric Series and Applications,}
American Mathematical Society, Providence, RI, 1988.

\bibitem{guillera}
J.~Guillera,
{\em A family of Ramanujan--Orr formulas for $1/\pi$,}
Integral Transforms and Special Functions, {\bf 26} (2015), no. 7, 531--538.

\bibitem{translation}
J.~Guillera and W.~Zudilin,
{\em Ramanujan-type formulae for $1/\pi$: The art of translation,}
in: {\em The Legacy of Srinivasa Ramanujan,}
B.~C.~Berndt and D.~Prasad (eds.), Ramanujan Math. Soc. Lecture Notes Series, {\bf 20} (2013), 181--195.

\bibitem{aeqb}
M.~Petkovsek, H.~Wilf and D.~Zeilberger,
$A = B$,
A. K. Peters, Wellesley, 1996.

\bibitem{ramanujan_pi}
S.~Ramanujan,
{\em Modular equations and approximations to $\pi$,}
Quart. J. Math (Oxford), {\bf45} (1914), 350--372;
Reprinted in \cite[pp.~23--39]{RCW}.

\bibitem{RCW}
S.~Ramanujan,
{\em Collected Papers,} Third printing,
AMS Chelsea, Providence, Rhode Island, 2000.

\bibitem{rogers}
M.~Rogers,
{\em New ${}_5F_4$ hypergeometric transformations, three-variable Mahler measures,
and formulas for $1/\pi$,}
Ramanujan J., {\bf 18} (2009), 327--340.

\bibitem{rogersstraub}
M.~Rogers and A.~Straub,
{\em A solution of Sun's {\em\$520} challenge concerning $520/\pi$,}
Int. J. Number Theory, {\bf 9} (2013), 1273--1288.

\bibitem{sloane}
N.~J.~A.~Sloane,
{\em The On-Line Encyclopedia of Integer Sequences,}
published electronically at \url{http://oeis.org}, 2015.

\bibitem{sun}
Z.-W.~Sun,
{\em List of conjectural series for powers of $\pi$ and other constants,}
arXiv:1102.5649v47, December 2014.

\bibitem{wanphd}
J.~G.~Wan,
{\em Random walks, elliptic integrals and related constants,}
Ph.\,D.~thesis, University of Newcastle, Australia (2013).

\bibitem{wan}
J.~G.~Wan,
{\em Series for $1/\pi$ using Legendre's relation,}
Integral Transforms and Special Functions {\bf 25} (2014) no. 1, 1--14.

\bibitem{wanzudilin}
J.~G.~Wan and W.~Zudilin,
{\em Generating functions of Legendre polynomials: a tribute to Fred Brafman,}
J. Approx. Theory {\bf 164} (2012), 488--503.

\bibitem{lawrence}
D.~Ye,
{\em Private communication to S. Cooper,} November 11, 2013.

\bibitem{zudilinBAMS}
W.~Zudilin,
{\em A generating function of the squares of Legendre polynomials,}
Bull. Aust. Math. Soc. {\bf 89} (2014), 125--131.

\end{thebibliography}
\end{document}